\documentclass[10pt]{amsart}

\usepackage{amsmath,amsthm}
\usepackage{amsfonts}
\usepackage{amssymb}
\usepackage{enumerate}

\usepackage[usenames]{color}

\newcommand{\vol}{\operatorname{vol}}
\newcommand{\inj}{\operatorname{inj}}

\newcommand{\reals}{\mathbb{R}}

\newcommand{\dvol}{\operatorname{dvol}}

\newcommand{\tr}{\operatorname{tr}}
\newcommand{\II}{\operatorname{II}}
\newcommand{\graph}{\operatorname{graph}}

\newcommand{\proj}{\operatorname{proj}}

\newtheorem{theorem}{Theorem}[section]
\newtheorem{lemma}[theorem]{Lemma}

\newtheorem{proposition}[theorem]{Proposition}
\theoremstyle{definition}
\newtheorem{definition}{Definition}[section]
\newtheorem{question}{Question}

\newtheorem{remark}{Remark}
\numberwithin{equation}{section}

\begin{document}

\title[Singular Time of the Mean Curvature Flow]{A Characterization of the Singular Time of the Mean Curvature Flow}
\author{Andrew A. Cooper}
\address{Department of Mathematics\\
Michigan State University\\
East Lansing, Michigan 48824}
\email{andrew.a.cooper@gmail.com}
\thanks{The author was partially supported by an RTG Research Training in Geometry and Topology NSF grant DMS 0353717 and as a graduate student on NSF grant DMS 06-04759.}
\date{}


\begin{abstract}
In this note we investigate the behaviour at finite-time singularities of the mean curvature flow of compact Riemannian submanifolds $M_t^m\hookrightarrow (N^{m+n},h)$. We show that they are characterized by the blow-up of a trace $A=H\cdot\II$ of the square of the second fundamental form.
\end{abstract}

\maketitle

\section{Introduction}

It is well known that the mean curvature flow $\partial_t F=H$ of submanifolds \linebreak $\displaystyle{F_t:M^m\hookrightarrow\reals^{m+n}}$ has finite-time singularities characterized by the blowup of the second fundamental form $\II$:

\begin{theorem}[Huisken \cite{huisken84}]\label{huisken} Suppose $T<\infty$ is the first singular time for a compact mean curvature flow.  Then $\max_{M_t}|\II|\rightarrow \infty$ as $t\rightarrow T$.\end{theorem}

We will prove that in fact it suffices to consider the tensor $A_{ij}=H^\alpha h_{ij\alpha}$, where $H=\tr\II$ is the mean curvature and $h$ are the components of $\II$:
 \begin{theorem}\label{main} Let $(N,h)$ be a Riemannian manifold with bounded geometry. Suppose $T<\infty$ is the first singular time for a mean curvature flow of compact submanifolds of $(N,h)$.  Then $\max_{M_t}|A|\rightarrow \infty$ as $t\rightarrow T$.\end{theorem}
 
By slightly modifying the proof of Theorem \ref{main}, we also obtain
\begin{theorem}\label{blowup} Suppose that along the flow, $|\II(x,t)|^p(T-t)\leq C$ for some $p\in(1,2]$.  Then $\max_{M_t}|H|\rightarrow \infty$ as $t\rightarrow T$.\end{theorem}

\section{Preliminaries}
First we recall some evolution equations for the flow.  We use indices \linebreak$\displaystyle{1\leq i,j,k,l\leq m}$, $\displaystyle{m+1\leq \alpha,\beta,\gamma\leq m+n}$. $h_{ij\alpha}$ denotes the $\alpha$th component of $\II(\partial_i,\partial_j)$.  $H^\alpha$ denotes the $\alpha$th component of the mean curvature $H$.  $g_{ij}$ denotes the induced metric on $M$.  $\overline{R}$ with four indices denotes the Riemannian curvature of $(N,h)$, and $\overline{R}$ with two indices denotes the Ricci curvature of $(N,h)$.  $\nabla_i$ denotes the tangential covariant derivative in the direction $i$. $\overline{\nabla}$ denotes the covariant derivative of $h$.  We use the summation convention on upper and lower indices.

\begin{lemma}[Huisken \cite{huisken84}, Wang \cite{wang01}] Along a mean curvature flow \linebreak $M_t\hookrightarrow (N,h)$, we have
	\begin{enumerate}\item $\partial_t g_{ij}=-2H^\alpha h_{ij\alpha}$
	\item $\partial_t \dvol = -|H|^2\dvol$
	\item \begin{equation*}\begin{split}\partial_t h_{ij\alpha}=&\Delta h_{ij\alpha} + (\overline{\nabla}_k\overline{R})_{\alpha ij}^{\ \ \ k}+(\overline{\nabla}_j\overline{R})_{\alpha k i}^{\ \ \ k}-2\overline{R}_{lijk}h_{\alpha}^{\ l k}\\
	&+2\overline{R}_{\alpha\beta jk}h^{k\ \beta}_{\ i}+2\overline{R}_{\alpha\beta ik}h^{k\ \beta}_{\ j}-\overline{R}_{lki}^{\ \ \ k}h_{j\ \alpha}^{\ l}-\overline{R}_{lkj}^{\ \ \ k}h_{i \ \alpha}^{\ l}+\overline{R}_{\alpha k\beta}^{\ \ \ k}h^{\ \ \beta}_{ij}\\
	&+ h_{il\alpha}(h_{lk\gamma}h_{j}^{\ k\gamma}-h_{kj\gamma}H^\gamma)+ h_{lk\alpha}(h^{lk}_{\ \ \gamma}h_{ij}^{\ \ \gamma}-h_{lj\gamma}h_i^{\ k\gamma})\\
	&+ h_{ik\beta}(h_l^{\ k\beta}h_{\alpha\ j}^{\ l}-h_{lj}^{\ \ \beta}h_{\alpha}^{\ lk})-h_{\alpha jk}h_{\beta i}^{\ \ k}H^\beta+h_{ij\beta}\langle e_\beta,\overline{\nabla}_He_\alpha\rangle\end{split}\end{equation*}
	\end{enumerate}\end{lemma}

By integrating the evolution equation for $|\nabla^s\II|$ and using the H\"{o}lder and Morrey inequalities, one can obtain 

\begin{theorem}[Huisken \cite{huisken84}]\label{bounds} Along the mean curvature flow, \linebreak $\displaystyle{\sup_{M\times [0,T)}|\nabla^s\II|}$ is bounded in terms of $\displaystyle{\sup_{M\times [0,T)}|\II|}$ and the ambient geometry bounds.\end{theorem}

We recall a few lemmas about one-parameter families of Riemannian metrics:
\begin{lemma}[Glickenstein \cite{glickenstein03}]\label{glickenstein} Suppose a one-parameter family of complete Riemannian manifolds $(M, g(t))$ is uniformly continuous in $t$, that is, for any $t_0$ and any $\epsilon>0$ there is $\delta>0$ so that $(1-\epsilon)g(t_0)\leq g(t)\leq (1+\epsilon)g(t_0)$ for $t\in [t_0,t_0+\delta]$.  Then for any $p\in M$, $r>0$, the metric balls centred at $p$ satisfy:
	\begin{equation*}\begin{split}B_{g(t_0)}(p,\frac{r}{\sqrt{1+\epsilon}})\subseteq B_{g(t)}(p,r)\subseteq B_{g(t_0)}(p,\frac{r}{\sqrt{1-\epsilon}})\end{split}\end{equation*}\end{lemma}
\begin{proof} Let $p,q\in M$.  Let $\gamma:[0,S]\rightarrow M$ be a minimising geodesic from $p$ to $q$ for the metric $g(t_0)$.  Then the distance $d_{g(t_0)}(p,q)$ in the metric $g(t_0)$ satisfies
	\begin{equation}\begin{split}d_{g(t_0)}(p,q)=&\int_0^S |\dot{\gamma}|_{g(t_0)}(s)ds\\
	\geq& \frac{1}{\sqrt{1+\epsilon}}\int_0^S|\dot{\gamma}|_{g(t)}(s)ds\\
	\geq & \frac{1}{\sqrt{1+\epsilon}}d_{g(t)}(p,q)\end{split}\end{equation}
	so that $\frac{1}{\sqrt{1+\epsilon}}d_{g(t)}(p,q) \leq d_{g(t_0)}(p,q)$.  This immediately implies  \begin{equation}\begin{split}B_{g(t_0)}(p,\frac{r}{\sqrt{1+\epsilon}})\subset B_{g(t)}(p,r)\end{split}.\end{equation}
	The other inclusion is analogous.\end{proof}
\begin{lemma}[Hamilton \cite{hamilton82}]\label{metricequiv}Let $(M,g(t))$ be a one-parameter family of compact Riemannian manifolds defined for $t\in[0,T)$.  Suppose that\begin{align*}\int_0^T\max_{M_t}|\frac{\partial g}{\partial t}|_{g(t)}dt<\infty\end{align*}
	Then the metrics $g(t)$ are uniformly equivalent and converge pointwise as $t\rightarrow T$ to a continuous positive-definite metric $g(T)$.\end{lemma}

\section{$\II$ and the Injectivity Radius}

We will prove Theorems \ref{main} and \ref{blowup} by a blow-up argument.  In particular we will use the Cheeger-Gromov convergence theorem to extract a limit of some submanifolds $F_j:M\hookrightarrow\reals^{m+n}$, thought of as Riemannian manifolds $(M,F_j^*dx^2)$.  We therefore need the following relationship between injectivity radius and the second fundamental form.

\begin{theorem}\label{injectivity}Let $F:M^m\looparrowright \reals^{m+n}$ be an immersion with $|\II|\leq C$.  Then $\inj(M,F^*dx^2)\geq \frac{1}{2\sqrt{2}C}$.\end{theorem}

We begin by considering the case of the graph of a map $\psi:\reals^m\rightarrow\reals^n$, as in \cite{reilly73}.  We need to compare the standard square-norm of certain objects, e.g. $|D^2\psi|^2=\sum_{\substack{1\leq \alpha\leq n\\1\leq i,j\leq m}}\left(\frac{\partial^2\psi_\alpha}{\partial x_i\partial x_j}\right)^2$, with the norms of the tensors $\II$ and $\nabla\II$ in the metric $g$ induced by the immersion. To keep the norms straight, in this section we use $|\cdot|$ for the standard square-norm and $|\cdot|_g$ for the norm in $g$:\begin{equation}\begin{split}
	|\II|_g^2=&h_{ij\alpha}h_{kl\beta}g^{\alpha\beta}g^{ik}g^{jl}\\
	|\nabla\II|^2_g=&\nabla_ih_{jk\alpha}\nabla_ph_{qr\beta}g^{ip}g^{jq}g^{kr}g^{\alpha\beta}\end{split}\end{equation}

\begin{lemma}\label{hessianbound}Let $\psi:D_r^m\rightarrow \reals^n$ be a $C^2$ function on the disc of radius $r$.  Then\begin{align*}
	|D^2\psi|^2\leq (1+|D\psi|^2)^3|\II|_g^2\end{align*}where $\II$ is the second fundamental form of the graph of $\psi$.\end{lemma}
\begin{proof}
	The graph of $\psi$ has immersion map $F(x_1,\ldots,x_m)=(x_1,\ldots,x_m,\psi_1,\ldots,\psi_n)$. We use the following tangent and normal frames, where $1\leq i\leq m$ and $1\leq \alpha\leq n$:\begin{equation}\begin{split}
		e_i=&(0,\ldots,0,1,0,\ldots,0,\frac{\partial\psi_1}{\partial x_i},\ldots,\frac{\partial\psi_n}{\partial x_i})=(0,\ldots,0,1,0,\ldots,0,D_i\psi)\\
		\nu_\alpha=&(-\frac{\partial\psi_\alpha}{\partial x_1},\ldots,-\frac{\partial \psi_\alpha}{\partial x_m},0,\ldots,0,1,0,\ldots,0)=(-D\psi_\alpha,0,\ldots,0,1,0,\ldots,0)\end{split}\end{equation}
	These choices induce the metric on the tangent bundle of the graph, which we denote by $g$ with Latin indices:
		\begin{equation}
		g_{ij}=e_i\cdot e_j=\delta_{ij}+D_i\psi\cdot D_j\psi\end{equation}
	We also get a metric on the normal bundle, which we denote by $g$ with Greek indices:\begin{equation}
		g_{\alpha\beta}=\nu_\alpha\cdot\nu_\beta=\delta_{\alpha\beta}+D\psi_\alpha\cdot D\psi_\beta\end{equation}
	We will use $g^{ij}$ to denote the inverse matrix to $g_{ij}$ and $g^{\alpha\beta}$ to denote the inverse to $g_{\alpha\beta}$.
	We compute the second fundamental form. Note that $D^2F=(0,D^2\psi)$. So we have\begin{equation}\begin{split}
		\II(e_i,e_j)=&\proj^\perp(D^2F(e_i,e_j))\\
		=&(D_{ij}^2F\cdot\nu_\beta)g^{\alpha\beta}\nu_\alpha\\
		=&\frac{\partial^2\psi_\beta}{\partial x_i\partial x_j}g^{\alpha\beta}\nu_\beta\end{split}\end{equation}
	In components, $h_{ij\alpha}=\frac{\partial^2\psi_\alpha}{\partial x_i\partial x_j}$.

	Then the norm-squared of the second fundamental form is\begin{equation}\begin{split}
		|\II|_g^2=&\frac{\partial^2\psi_\alpha}{\partial x_i\partial x_j}\frac{\partial^2 \psi_\beta}{\partial x_k\partial x_l}g^{\alpha\beta}g^{ik}g^{jl}\end{split}.\end{equation}
	We can think of $|\II|_g^2$ as the norm-squared of $D^2\psi$ in the metric $g$ as opposed to the standard metric.  We will compare $g^{\alpha\beta}$ and $g^{ij}$ to the standard metric by giving estimates for the eigenvalues of $g^{\alpha\beta}$ and $g^{ij}$. To do this we estimate the eigenvalues of $g_{ij}$ and $g_{\alpha\beta}$.

	Each eigenvalue  $\lambda$ of $g_{\alpha\beta}$ has the form $\lambda=g(X,X)=g_{\alpha\beta}X^\alpha X^\beta$ for some eigenvector $X\in\reals^n$ with $|X|^2=\sum_{\alpha}(X^\alpha)^2=1$.  Then\begin{equation}\begin{split}
		\lambda=&(\delta_{\alpha\beta}+D\psi_\alpha\cdot D\psi_\beta)(X^\alpha X^\beta)\\
			=&|X|^2 + (X^\alpha D\psi_\alpha)\cdot (X^\beta D\psi_\beta)\\
			=&1 + |X^\alpha D\psi_\alpha|^2\end{split}\end{equation}
	Thus $1\leq \lambda\leq 1+|D\psi|^2$. Similarly for an eigenvalue $\mu$ of $g_{ij}$, we have\begin{equation}\begin{split}
		\mu=&(\delta_{ij} + D_i\psi\cdot D_j\psi)X^iX^j\\
		=&|X|^2 + |D_X\psi|^2\\
		=&1+|D_X\psi|^2\end{split}\end{equation}
	So $1\leq \mu\leq 1+|D\psi|^2$.
		
	Thus the eigenvalues of the inverse matrices $g^{\alpha\beta}$ and $g^{ij}$ are bounded away from zero and infinity:\begin{equation}\label{inveigen}\begin{split}
		1 \geq |\lambda^{-1}|\geq& \frac{1}{1+|D\psi|^2}\\
		1 \geq |\mu^{-1}|\geq& \frac{1}{1+|D\psi|^2}\end{split}\end{equation}
	So we can estimate\begin{equation}\begin{split}
		|\II|_g^2=&\frac{\partial^2\psi_\alpha}{\partial x_i\partial x_j}\frac{\partial^2 \psi_\beta}{\partial x_k\partial x_l}g^{\alpha\beta}g^{ik}g^{jl}\\
		\geq & \sum_{\substack{1\leq\alpha\leq n\\1\leq i,j\leq m}}\left(\frac{\partial^2\psi_\alpha}{\partial x_i\partial x_j}\right)^2\frac{1}{(1+|D\psi|^2)(1+|D\psi|^2)^2}\\
		= & |D^2\psi|^2\frac{1}{(1+|D\psi|^2)^3}\end{split}\end{equation}
	which establishes our lemma.
		
\end{proof}

Our next step is to show that any immersed submanifold $F:M\looparrowright\reals^{m+n}$ can be written as a collection of graphs of functions $\psi$ with small $|D\psi|$.

We introduce the following notation and notions, following \cite{langer85}.  Given $q\in M$, denote by $A_q$ any Euclidean isometry which takes $F(q)$ to the origin and $T_{F(q)}F(M)$ to the plane $\{(x_1,\ldots,x_m,0)\}$.  Let $\pi$ be the projection of $\reals^{n+m}$ to the plane $\{(x_1,\ldots,x_m,0)\}$.  Define $U_{r,q}\subset M$ to be the component of $(\pi\circ A_q^{-1}\circ F)^{-1}(D_r)$ which contains $q$.  We call $F:M\looparrowright \reals^{n+m}$ a $(r,\alpha)$-immersion if for each $q\in M$ there is some $\psi_q:D_r^m\rightarrow\reals^n$ with $|D\psi_q|\leq \alpha$ so that $A_q^{-1}\circ F(U_{r,q})=\graph(\psi_q)$.

\begin{lemma}\label{arralpha} Let $0<\alpha\leq 1$. Then for any $C^2$-immersed submanifold $F:M^m\looparrowright \reals^{n+m}$ and any $r$ satisfying\begin{align*}
	r\leq \frac{\alpha}{(1+\alpha^2)^{3/2}}\frac{1}{\sup_M|\II|_g}\end{align*}
	$F$ is a $(r,\alpha)$-immersion.\end{lemma}
\begin{proof}
	Let $q\in M$ be arbitrary.  Every submanifold is locally a graph over its tangent plane; thus $A_q(F(U_{r,q}))$ can be written as a graph over $D_r$ for small enough $r$.  So we set $S_q=\sup\{r|F(U_{r,q})=\graph(\psi_{r,q})\}$.  For any large $K$, if $F(U_{r,q})=\text{graph}(\psi_{r,q})$ and $|D\psi_{r,q}|\leq \frac{K}{2}$, we can extend $\psi_{r,q}$ to have a larger domain and still $|D\psi|\leq K$.  Thus we have\begin{equation}\begin{split}
		\lim_{r\rightarrow S_q}\inf_{\psi}\sup_{D_r}|D\psi|=\infty\end{split}\end{equation}
	where the infimum is taken over all $\psi$ of which $F(U_{r,q})$ is a graph.
	Thus for our given $\alpha$ there exists some $r_q, \psi_{q}:D_{r_q}\rightarrow \reals^n$ with $\sup_{D_{r_q}}|D\psi_{q}|=\alpha$.  Now we use the fundamental theorem of calculus and Lemma \ref{hessianbound} to get\begin{equation}\begin{split}
		\alpha=\sup_{D_{r_q}}|D\psi_{q}|\leq r\sup_{D_{r_q}}|D^2\psi_{q}|\leq r_q (1+\alpha^2)^{3/2}\sup_{D_{r_q}}|\II_{\psi_q}|_g\end{split}\end{equation}
	which implies that \begin{equation}\label{Sq}
		r_q\geq \frac{\alpha}{(1+\alpha^2)^{3/2}}\frac{1}{\sup_{D_{r_q}}|\II_{\psi_q}|_g}\geq\frac{\alpha}{(1+\alpha^2)^{3/2}}\frac{1}{\sup_M|\II|_g}.\end{equation}
	So for $r$ less than the right-hand side of (\ref{Sq}), there is some $\psi:D_r\rightarrow\reals^{m+n}$ which makes $F(U_{r,q})$ a graph and $|D\psi|\leq \alpha$.
	\end{proof}
	
\begin{proof}[Proof of Theorem \ref{injectivity}]
	Now choose $\alpha=1$, and let $r$ be given by Lemma \ref{arralpha}.  It is clear that $B(q,r)\subset U_{r,q}$, since $A_q F(U_{r,q})$ is a graph over a disc of radius $r$.  Thus $B(q,r)$ can be written as a graph over the tangent plane, and in particular $\inj(q)\geq r$.
	
	Since $q$ was arbitrary, we have $\inj(M)\geq r>0$.
\end{proof}

\section{The Tensor $A$ Blows Up}
We will prove Theorem \ref{main} by contradiction.  To this end, assume $\displaystyle{\max_{M}|A(t)|\leq C}$ for all $t\in[0,T)$, and that the flow has a singularity at $T<\infty$.

In particular, we have $|H|^4=(\tr A)^2\leq n|A|^2\leq nC^2$.  So $|H|$ is also bounded along the flow.  We will use $C$ to denote this bound as well.

By Theorem \ref{huisken}, we know that as $t\rightarrow T$, $\displaystyle{\max_{M}|\II(t)|}\rightarrow \infty$.  Let $(p_j,t_j)$ be a sequence in $M\times [0,T)$ with $t_j\rightarrow T$  and $\displaystyle{\max_{t\leq t_j}|\II|=|\II(p_j,t_j)|=:Q_j\rightarrow \infty}$.  

Since $|\partial_tF|=|H|\leq C$, we know that $F_t(M)$ is contained in the $CT$ tubular neighborhood of $F_0(M)$. Thus the $F(p_j,t_j)$ accumulate.  Passing to a subsequence, we have $F(p_j,t_j)\rightarrow p_0$. For any $R>0$, we may choose a $j_0$ so that $F(p_j,t_j)$ lies in the ambient ball of radius $R$ about $p_0$, $B^N(R)$ for all $j\geq j_0$.  In particular, we will take $R$ to be less than the injectivity radius of $(N,h)$.
  
Consider the flows given by scaling the ambient metric by $Q_j$ and time by $Q_j^{-2}$:\begin{equation}\begin{split}
	F_j(p,t)=F(p,t_j+\frac{t}{Q_j^2}):M\hookrightarrow (B^N,Q_j^2h)\end{split}\end{equation}
\begin{lemma}\label{bound} Each $F_j$ is  a mean curvature flow  on $M\times[-Q_j^2t_j,0]$.  The second fundamental form of $F_j$ is bounded:\begin{equation}\begin{split}\max_{t\leq 0}|\II_j|=|\II_j(p_j,0)|=1\end{split}\end{equation}\end{lemma}
\begin{proof} Clearly $\partial_tF_j=Q_j^{-2}\partial_t F$.  We need to show that by scaling the ambient metric, we induce the same scaling in $H$.  By definition\begin{equation}\begin{split}
	H_{Q_j^2h}(F_j)&=\tr_{Q_j^2h}\II_{Q_j^2h}(F_j)\\
	&=(Q_j^2h)^{pq} (\partial_p\partial_qF_j)^{\perp_{Q_j^2h}}\end{split}\end{equation}
	where $(Q_j^2h)^{pq}$ is the inverse matrix of $(Q_j^2h)(\partial_pF_j,\partial_qF_j)$ and $\perp_{Q_j^2h}$ is the projection onto the normal bundle induced from $Q_j^2h$.
	
	$Q_j^2h$ induces the same splitting into tangent and normal bundles as $h$, so we have
	\begin{equation}\begin{split}
	H_{Q_j^2h}(F_j)&=Q_j^{-2}\left[h^{pq} (\partial_p\partial_qF_j)^{\perp_{h}}\right]\\
	&=Q_j^{-2}\tr_h\II_{h}(F_j)\\
	&=Q_j^{-2}H_h(F_j)\end{split}\end{equation}
	So $H$ scales as required.
	
	Similarly scaling the ambient metric by $Q_j^2$ scales $|\II|$ by $Q_j^{-1}$, so we have \begin{equation}\begin{split}
	\max_{t\leq 0}|\II_j|&=\max_{t\leq t_j}Q_j^{-1}|\II|\\
	&= Q_j^{-1}|\II(p_j,t_j)|=1\end{split}\end{equation}\end{proof}
	
It is clear that the $(B^N,Q_j^2h,p_0)$ converge in the Cheeger-Gromov sense to $(\reals^{m+n},dx^2,0)$, where $dx^2$ is the Euclidean metric.  In particular, we have a monotone exhausting sequence of open sets $V_j\subset \reals^{m+n}$ and embeddings \linebreak $\displaystyle{\psi_j:(V_j,0)\rightarrow (B^N,p_0)}$, such that $\psi_{j}^*(Q_j^2h)\rightarrow dx^2$.  

Let $s_0=-Q_1^2t_1$.  After passing to a smaller spatial region $\tilde{M}\subset M$, we can assume $F_j(\tilde{M}\times [s_0,0])\subset \psi_j(V_j)$.   We restrict our argument to this smaller region and write $M$ without confusion.

Define $\tilde{F}_j:M\times [s_0,0]\rightarrow \reals^{m+n}$ by $\tilde{F_j}=\psi_j^{-1}F_j$.  Each $\tilde{F}_j$ is a mean curvature flow with respect to the metric $\psi_j^*(Q_j^2h)$.

The second fundamental forms $\tilde{\II}_j$ of the $\tilde{F}_j$ are uniformly bounded, so Theorem \ref{bounds} gives uniform bounds on the covariant derivatives of the $\tilde{\II}_j$.  

Since $\partial_t \tilde{F}_j=\tilde{H}_j$, we get bounds on the time derivative of $\tilde{F}_j$.  In fact the evolution of $H$ gives a  bound \begin{equation}\begin{split}
	|\partial_t^2\tilde{F}_j|=|\partial_t\tilde{H}_j|&\leq |\Delta \tilde{H}_j| + C_1|\nabla \tilde{H}_j|+C_2|\tilde{H}_j||\tilde{\II}_j|^2\\
	&\leq |\nabla^2\tilde{\II}_j|+C_1|\nabla\tilde{\II}_j|+C_2|\tilde{\II}_j|^3\end{split}\end{equation}
Similarly, any iterated time derivative $\partial_t^s\tilde{F}_j=\partial_t^{s-1}(\tilde{H}_j)$  is controlled in terms of $|\nabla^r\tilde{\II}_j|$ for $r\leq 2(s-1)$.  The mixed derivatives $\partial_t^r\nabla^s\tilde{\II}_j$ are similarly controlled by $|\nabla^l\tilde{\II}_j|$ for $l\leq 2r+s$.

Since $Q_j^2h\rightarrow dx^2$ in $C^k$ for any $k$, our bounds on the $|\nabla^s\tilde{\II}_j|$ give bounds on $|\overline{\nabla}^s\overline{\II}_j|$, where $\overline{\nabla}$ and $\overline{\II}_j$ are the connection and second fundamental form of $\tilde{F}_j$ with respect to the metric $dx^2$.

Let $t\in[s_0,0]$.  Theorem \ref{injectivity} gives $\inj(M,\tilde{F}_j(t)^*dx^2)\geq \frac{1}{2\sqrt{2}}$.  The Gauss equation guarantees that the Riemannian curvature at its covariant derivatives of $(M,\tilde{F}_j(t)^*dx^2,p_j)$ are all bounded uniformly in $j$.  Thus by Cheeger-Gromov there is a limit Riemannian manifold $(M_\infty, g_\infty(t),p_\infty)$.

Adapting the ideas of \cite{sesum03}, we consider the growth of balls in $(M_\infty, g_\infty(0))$.  We will write $g_\infty$ for $g_\infty(0)$.

\begin{proposition}\label{intrinsic}
	$(M, g_\infty)$ has euclidean volume growth about $p_\infty$.\end{proposition}
\begin{proof}
	Let us use the following conventions for balls and volume forms.  $B_\infty(\rho)$ will denote the metric ball in $g_\infty$ centered at $p_\infty$; $B_j(\rho)$ will denote the metric ball in $F_{t_j}^*(Q_j^2h)$ centred at $p_j$; $B_{t_j}(\rho)$ will denote the metric ball in $F_{t_j}^*h$ centred at $p_j$.  $\vol_\infty$ will denote the volume form of $g_\infty$; $\vol_{j}$ will denote the volume form of $F_{t_j}^*(Q_j^2h)$; $\vol_{t_j}$ will denote the volume form of $F_{t_j}^*h$.  Note that
		\begin{equation}\begin{split}B_j(\rho)&=B_{t_j}(\frac{\rho}{Q_j})\\
		\vol_j&=Q_j^m\vol_{t_j}\end{split}\end{equation}
		
	We have, for any $r>0$
	\begin{equation}\label{one}\begin{split}
 	\frac{\vol_\infty(B_\infty(r))}{r^m}&=\lim_j\frac{\vol_j(B_j(r))}{r^m}\\
 	&=\lim_j\frac{\vol_{t_j}(B_{t_j}(\frac{r}{Q_j}))}{(\frac{r}{Q_j})^m}\end{split}\end{equation}
 	
 	The evolution of $g$ is \begin{equation}\begin{split}
 		\partial_t g_{ij}=-2H^\alpha h_{ij\alpha}=-2A_{ij}\end{split}\end{equation}
 	so we have $|\partial_t g|\leq C$, and in particular $g$ is uniformly continuous in time in the sense of Lemma \ref{glickenstein}.
	
	Thus we may apply Lemma \ref{glickenstein} to estimate the metric balls at any time $t_j$ by the metric ball at time $t_{j_0}$, so long as $t_j-t_{j_0}\leq \delta$.  Since $t_j\rightarrow T$, we can pick a $j_0$ so that this condition holds for all $j\geq j_0$.   So we can estimate (\ref{one}) by:
	
	\begin{equation}\begin{split}\label{two}
	 	\lim_j\frac{\vol_{t_j}(B_{t_j}(\frac{r}{Q_j}))}{(\frac{r}{Q_j})^m}&\leq \lim_j\frac{\vol_{t_j}(B_{t_{j_0}}(\frac{r}{(\sqrt{1-\epsilon}Q_j})))}{(\frac{r}{Q_j})^m}\end{split}\end{equation}
	 
	 The evolution of the volume form shows that the flow is pointwise volume-reducing. So $\vol_{t_j}\leq \vol_{t_{j_0}}$ for $j\geq j_0$.  Thus we can estimate (\ref{two}) by 

	 \begin{equation}\label{volest}\begin{split}
	 \lim_j\frac{\vol_{t_j}(B_{t_{j_0}}(\frac{r}{(\sqrt{1-\epsilon}Q_j})))}{(\frac{r}{Q_j})^m}&\leq \lim_j\frac{\vol_{t_{j_0}}B_{t_{j_0}}(\frac{r}{\sqrt{1-\epsilon}Q_j})))}{(\frac{r}{Q_j})^m}\\
	 &=(1-\epsilon)^{-\frac{m}{2}}\lim_j\frac{\vol_{t_{j_0}}(B_{t_{j_0}}(\frac{r}{\sqrt{1-\epsilon}Q_j}))}{(\frac{r}{\sqrt{1-\epsilon}Q_j})^m}\end{split}.\end{equation}
	 The only dependence of the right hand side on $j$ is in the $Q_j$.
	 
	 The limit on the right hand side of (\ref{volest}) is the local volume comparison at $p_{j_0}$ for the Riemannian manifold $(M, F_{t_{j_0}}^*h)$. It is well-known that this limit is $\omega_m$, the volume of the Euclidean unit $m$-ball.  Therefore we have
	 
	 \begin{equation}\begin{split}
	 \frac{\vol_\infty(B_\infty(r))}{r^m}\leq (1-\epsilon)^{-\frac{m}{2}}\omega_m\end{split}\end{equation}
	 Since $\epsilon$ was arbitrary, we have shown $\vol_\infty(B_\infty(r))\leq \omega_mr^m$.
	 
	 To show the reverse inequality, we make a similar argument starting from (\ref{one}), this time using the first inclusion of Lemma \ref{glickenstein}.  We now seek to estimate $\vol_{t_j}$ {\em below} by $\vol_{t_{j_0}}$.  Since we have assumed $|H|\leq C$, the evolution of $\vol$ implies that \begin{equation}\begin{split}\vol_{t_j}\geq e^{-C^2(t_j-t_{j_0})}\vol_{t_{j_0}}\end{split}\end{equation}
	 and taking $j_0$ large enough we may ensure that $e^{-C^2(t_j-t_{j_0})}\geq 1-\epsilon$.  Then we can estimate (\ref{one}) by
	 
	 \begin{equation}\begin{split}\label{three}
	 	\lim_j\frac{\vol_{t_j}(B_{t_j}(\frac{r}{Q_j}))}{(\frac{r}{Q_j})^m}\geq \lim_j (1-\epsilon)(1+\epsilon)^{-\frac{m}{2}}\frac{\vol_{t_{j_0}}(B_{t_{j_0}}(\frac{r}{\sqrt{1+\epsilon}Q_j}))}{(\frac{r}{\sqrt{1+\epsilon}Q_j})^m}\end{split}\end{equation}
	 Again we can take the limit in $j$ to get\begin{equation}\begin{split}
	 	\frac{\vol_\infty(B_\infty(r))}{r^m}\geq (1-\epsilon)(1+\epsilon)^{-\frac{m}{2}}\omega_m\end{split}\end{equation}
	 Since $\epsilon$ was arbitrary we have shown $\vol_\infty(B_\infty(r))\geq \omega_mr^m$.
	 \end{proof}

\begin{proof}[Proof of Theorem \ref{main}]
	To finish the proof of the theorem, we want to use the volume growth of $(M_\infty,g_\infty)$ to obtain a contradiction.  The expansion for the volume of balls about $p$ in $r$ is \begin{equation}\begin{split}
		\vol_\infty(B_\infty(r))=\omega_mr^m(1-\frac{R_\infty(p)}{6(m+2)}r^2 + O(r^3))\end{split}\end{equation}
	where $R_\infty(p_\infty)$ is the scalar curvature \cite{gallothulinlafontaine}. So Proposition \ref{intrinsic} immediately implies that $R_\infty(p_\infty)=0$.
	
	On the other hand, tracing the Gauss equation twice gives that\begin{equation}\begin{split}
		R_\infty(p_\infty)=&\lim_jR_j(p_j)\\
			 =& \lim_j |H_j(p_j)|^2-|\II_j(p_j)|^2\\
			 \leq&\lim_j \frac{C}{Q_j^2}-1\\
			 =&-1\end{split}\end{equation}
	 This is the desired contradiction.\end{proof}

\section{A Condition for the Blow Up of $H$}
	\begin{definition} A singularity at time $T<\infty$ is of {\em type I} if $|\II(x,t)|^2(T-t)\leq C<\infty$ for all $x\in M_t$ and all $t\in [0,T)$.\end{definition}
	This is the slowest possible rate of singularity formation, and is attained in the case of a shrinking sphere or cylinder. 
	
	We can prove that the mean curvature blows up under a slightly more general condition, namely
	\begin{theorem}Suppose that along the flow, $|\II(x,t)|^p(T-t)\leq C$ for some $p\in(1,2]$.  Then $\max_{M_t}|H|\rightarrow \infty$ as $t\rightarrow T$.\end{theorem}
	\begin{proof} For the purpose of contradiction, suppose $|H|\leq C$ all along the flow.
	
	As in the proof of Theorem \ref{main}, we consider the parabolic rescales \begin{equation}
		F_j(p,t)=F(p,t_j+\frac{t}{Q_j^2}):M\hookrightarrow (N, Q_j^2h)\end{equation}
	where $Q_j=\max_{t\leq t_j}|\II|$.  Each $F_j$ is a mean curvature flow, and we have \begin{equation}|H_j|\leq \frac{C}{Q_j}.\end{equation}  As in the proof of Theorem \ref{main}, we want to obtain a limit manifold $(M_\infty,g_\infty)$ whose volume growth yields a contradiction. 
	
	To proceed to a contradiction as in the proof of Theorem \ref{main}, we need to establish that the metrics $g(t)$ are uniformly continuous in time in the sense of Lemma \ref{glickenstein}.
	
	Consider the evolution of the metric:\begin{equation}\label{est}\begin{split}
		|\frac{\partial g}{\partial t}|=2|A|&\leq 2|H||\II|\\
		&\leq 2C (T-t)^{-\frac{1}{p}}\end{split}\end{equation}
	
	We want to integrate the inequality (\ref{est}) to estimate $|g(t_1)-g(t_0)|_0$, where $|\cdot|_0$ is the norm on two-tensors induced from $g(0)$. Lemma \ref{metricequiv} together with (\ref{est}) imply that the metrics $g(t)$ are uniformly equivalent. Thus at the expense of a uniform constant (which we will absorb into $C$) we may estimate\begin{equation}\label{intest}\begin{split}
		|g(t_1)-g(t_0)|_0=|\int \frac{\partial g}{\partial t}dt|_0&\leq \int |\frac{\partial g}{\partial t}|_0 dt\\
		&\leq C\int|\frac{\partial g}{\partial t}|_t dt\\
		&\leq 2C \int (T-t)^{-\frac{1}{p}}dt\\
		&=\frac{2C}{1-\frac{1}{p}}\left((T-t_0)^{1-\frac{1}{p}}-(T-t_1)^{1-\frac{1}{p}}\right)\end{split}\end{equation}
		
	The function $(T-t)^{1-\frac{1}{p}}$ is uniformly continuous for $t\in[0,T)$, so we see that the metrics $g(t)$ are uniformly continuous in $t$ with respect to the norm $|\cdot|_0$.  Since the metrics are all uniformly equivalent, this implies uniform continuity in the sense of Lemma \ref{glickenstein}.
	
	Then the proof of the theorem proceeds just as in Theorem \ref{main}.\end{proof}

\begin{remark}The tensors $H$ and $A$ are either both zero or both nonzero. A bound on $A$ implies that $H$ is bounded.  Theorem \ref{main} says that $H$ cannot decay ``too fast" relative to $\II$. The following question presents itself:
\begin{question}\label{question} Suppose $T<\infty$ is the first singular time for a compact mean curvature flow.  Is it generally true that $\max_{M_t}|H|\rightarrow\infty$ as $t\rightarrow T$?
	\end{question}\end{remark}	
 
\section*{Acknowledgements}
The author wishes to thank his adviser Jon Wolfson for his help and suggestions.  He also wishes to thank Xiaodong Wang for suggesting this problem.

\bibliography{mcf}{}
\bibliographystyle{plain}

\end{document}